\documentclass[a4paper, reqno]{amsart}

\usepackage[UKenglish]{babel}
\usepackage{amsmath, amssymb, amsthm}
\usepackage[hidelinks]{hyperref}
\usepackage{xpatch}
\usepackage{tikz-cd}
\usepackage{enumitem}
\usepackage{theoremref}
\usepackage{color}

\usepackage[inline,final]{showlabels}
\showlabels{thlabel}

\setlist[description]{font=\normalfont\scshape}

\usepackage{imakeidx}
\makeindex[name=notes,title=Notes throughout the document (red text),columns=1]

\xpatchcmd{\proof}{\itshape}{\normalfont\bfseries}{}{}
\newtheoremstyle{repeat}{}{}{\itshape}{}{\bfseries}{.}{.5em}{#3, repeated}

\newtheorem{theorem}{Theorem}[section]
\newtheorem{proposition}[theorem]{Proposition}
\newtheorem{lemma}[theorem]{Lemma}
\newtheorem{corollary}[theorem]{Corollary}
\newtheorem{fact}[theorem]{Fact}

\theoremstyle{definition}
\newtheorem{definition}[theorem]{Definition}
\newtheorem{remark}[theorem]{Remark}
\newtheorem{convention}[theorem]{Convention}
\newtheorem{example}[theorem]{Example}
\newtheorem{question}[theorem]{Question}

\theoremstyle{repeat}
\newtheorem*{repeated-theorem}{Repeat}

\newcommand{\A}{\mathcal{A}}

\newcommand{\C}{\mathcal{C}}

\newcommand{\K}{\mathcal{K}}

\newcommand{\M}{\mathcal{M}}
\newcommand{\MM}{\mathfrak{M}}
\newcommand{\N}{\mathbb{N}}

\newcommand{\Set}{\mathbf{Set}}
\newcommand{\Mod}{\mathbf{Mod}}
\newcommand{\SubMod}{\mathbf{SubMod}}
\newcommand{\MetMod}{\mathbf{MetMod}}
\newcommand{\SubMetMod}{\mathbf{SubMetMod}}

\newcommand{\SubSet}{\mathbf{SubSet}}

\DeclareMathOperator{\dom}{dom}
\DeclareMathOperator{\tp}{tp}
\DeclareMathOperator{\gtp}{gtp}
\DeclareMathOperator{\colim}{colim}
\DeclareMathOperator{\Sub}{Sub}
\DeclareMathOperator{\Sgtp}{S_{gtp}}
\DeclareMathOperator{\density}{density}
\DeclareMathOperator{\Hom}{Hom}

\newcommand{\op}{{\textup{op}}}
\renewcommand{\phi}{\varphi}

\newcommand{\forkindep}[1][]{%
  \mathrel{
    \mathop{
      \vcenter{
        \hbox{\oalign{\noalign{\kern-.3ex}\hfil$\vert$\hfil\cr
              \noalign{\kern-.7ex}
              $\smile$\cr\noalign{\kern-.3ex}}}
      }
    }\displaylimits_{#1}
  }
}
\newcommand{\nonforkindep}[1][]{%
  \mathrel{
    \mathop{
      \vcenter{
        \hbox{\oalign{\noalign{\kern-.3ex}\hfil$\vert$\rlap{$'$}\hfil\cr
              \noalign{\kern-.7ex}
              $\smile$\cr\noalign{\kern-.3ex}}}
      }
    }\displaylimits_{#1}
  }
}
\newcommand{\indep}[2][]{%
  \mathrel{
    \mathop{
      \vcenter{
        \hbox{\oalign{\noalign{\kern-.3ex}\hfil$\vert$\hfil\cr
              \noalign{\kern-.7ex}
              $\smile$\cr\noalign{\kern-.3ex}}}
      }
    }\displaylimits_{#1}^{#2}
  }
}
\newcommand{\indepprime}[2][]{%
  \mathrel{
    \mathop{
      \vcenter{
        \hbox{\oalign{\noalign{\kern-.3ex}\hfil$\vert$\rlap{$'$}\hfil\cr
              \noalign{\kern-.7ex}
              $\smile$\cr\noalign{\kern-.3ex}}}
      }
    }\displaylimits_{#1}^{#2}
  }
}

\title{The Kim-Pillay theorem for Abstract Elementary Categories}
\author{Mark Kamsma}
\email[Mark Kamsma]{m.kamsma@uea.ac.uk}
\urladdr{https://markkamsma.nl}
\date{\today \\ \indent \emph{2020 Mathematics Subjects Classification}: Primary: 03C45; secondary: 03C48, 03C52, 03C66, 18C35}
 
\begin{document}

\begin{abstract}
We introduce the framework of AECats (abstract elementary categories), generalising both the category of models of some first-order theory and the category of subsets of models. Any AEC and any compact abstract theory (``cat'', as introduced by Ben-Yaacov) forms an AECat. In particular, we find applications in positive logic and continuous logic: the category of (subsets of) models of a positive or continuous theory is an AECat.

The Kim-Pillay theorem for first-order logic characterises simple theories by the properties dividing independence has. We prove a version of the Kim-Pillay theorem for AECats with the amalgamation property, generalising the first-order version and existing versions for positive logic.
\end{abstract}

\maketitle

\setcounter{tocdepth}{1}
\tableofcontents

\section{Introduction}
\label{sec:introduction}
For any complete first-order theory, Shelah's notion of dividing gives a ternary relation on subsets of models. Stable theories can be characterised by the properties this relation has. Lieberman, Rosick\'y and Vasey proved a category-theoretic version of this characterisation in \cite{lieberman_forking_2019}. Similarly, we can characterise simple theories using the \emph{Kim-Pillay theorem}, see \cite{kim_simple_1997}. The main result of this paper is a category-theoretic version of this theorem (\thref{thm:kim-pillay-category-theoretic}).

For a first-order theory $T$, the category of models of $T$ with elementary embeddings forms an accessible category, but accessible categories are more general. For example, there is Shelah's notion of AEC (abstract elementary class, see e.g. \cite{shelah_classification_2009}), which is a class of structures with a choice of embedding, satisfying a few axioms. Every AEC can naturally be seen as an accessible category. Other examples can be found by considering the category of models of some theory in another form of logic, such as positive logic and continuous logic (see e.g. \cite{poizat_positive_2018, ben-yaacov_positive_2003, ben-yaacov_model_2008}). There is also the concept of compact abstract theories, or cats, from \cite{ben-yaacov_positive_2003}, which in practice turn out to be accessible categories. Even then, accessible categories are more general, they are generally the category of models of some infinitary theory with homomorphisms as arrows, see \cite[Theorem 5.35]{adamek_locally_1994}. We define a specific kind of accessible category, an AECat, which still covers all the previously mentioned cases.

Even though the framework of AECats is very close to that of AECs, it is still more general. Some settings are hard to handle as AECs, but naturally fit the category-theoretic framework of AECats. For example, the class of metric models of a continuous theory in the sense of \cite{ben-yaacov_model_2008} is not an AEC, but it does form an AECat (see \thref{ex:continuous-logic}). Of course, there are \emph{metric AECs} as introduced in \cite{hirvonen_categoricity_2009}, but AECats provide a unifying approach.

Simplicity has already been studied separately for some of the aforementioned settings. For example, in AECs \cite{hyttinen_independence_2006} and in positive logic \cite{pillay_forking_2000}, or more generally, in cats \cite{ben-yaacov_simplicity_2003} and in homogeneous model theory \cite{buechler_simple_2003}. A few days after the first preprint of this paper became available online, another preprint appeared \cite{grossberg_simple-like_2020}, studying different aspects of simple-like independence relations in AECs. 

In \cite{lieberman_forking_2019}, the concept of an abstract independence relation on a category is introduced. They prove that there can be at most one stable such independence relation (similar to \thref{cor:stable-independence-relation-unique}). They define an independence relation as a collection of commutative squares. This has the benefit that it allows for a more category-theoretic study of the independence relation. For example, assuming transitivity of the independence relation, these squares form a category. In our approach we will define an independence relation as a relation on triples of subobjects (section \ref{sec:independence-relations}). We lose the nice way of viewing the independence relation as a category, but the benefit is that the calculus we get is more intuitive and easier to work with. Under some mild assumptions both approaches are essentially the same, in the sense that we can recover one from the other.

\vspace{\baselineskip}\noindent
\textbf{Main results.} We introduce the concept of an AECat (\thref{def:aecat}), generalising both the category of models of some first-order theory $T$ and the category of subsets of models of $T$. The framework of AECats can also be applied to positive logic (\thref{ex:category-of-models-positive-theory}), continuous logic (\thref{ex:continuous-logic}), quasiminimal excellent classes (\thref{ex:quasiminimal-excellent-classes}), AECs (\thref{ex:aec}) and compact abstract theories (\thref{ex:cat}).

We introduce the notion of isi-dividing (\thref{def:isi-dividing}), which is closely related to the usual notion of dividing (\thref{rem:isi-dividing-vs-dividing} and \thref{rem:isi-dividing-partial-answers}). We use this to prove a version of the Kim-Pillay theorem for AECats.
\begin{theorem}[Kim-Pillay theorem for AECats]
\thlabel{thm:kim-pillay-category-theoretic}
Let $(\C, \M)$ be an AECat with the amalgamation property, and suppose that $\forkindep$ is a simple independence relation. Let $A, B, C$ be subobjects of a model $M$. Then $A \indep[C]{M} B$ if and only if $\gtp(A, B, C; M)$ does not isi-divide over $C$.
\end{theorem}
The theorem implies canonicity of simple and stable independence relations.
\begin{corollary}[Canonicity of simple independence relations]
\thlabel{cor:unique-independence-relation}
On an AECat with the amalgamation property there can only be one simple independence relation.
\end{corollary}
\begin{corollary}[Canonicity of stable independence relations]
\thlabel{cor:stable-independence-relation-unique}
On an AECat with the amalgamation property there can only be one stable independence relation. More precisely, if $\forkindep$ is a stable independence relation and $\nonforkindep$ is a simple independence relation then $\forkindep = \nonforkindep$.
\end{corollary}
In an AECat we have no syntax, so we consider Galois types instead of syntactic types (section \ref{sec:galois-types}). For first-order logic, positive logic and continuous logic Galois types coincide with syntactic types, in the sense that two tuples have the same Galois type if and only if they have the same syntactic type.

Note that \thref{cor:stable-independence-relation-unique} is similar to \cite[Corollary 9.4]{lieberman_forking_2019} (see also \thref{rem:lrv-independence}), which is a corollary to their more general canonicity result Theorem 9.1. They only study independence relations that satisfy \textsc{Stationarity} (``uniqueness'' in their paper), which is of course not satisfied by general simple independence relations. On the other hand, we rely on what we call \textsc{Union} (which is similar to their ``$(< \aleph_0)$-witness property'').

\vspace{\baselineskip}\noindent
\textbf{Overview.} We start by setting up the framework of AECats in section \ref{sec:aecats}. The idea is that any category of models of some theory will fit this framework. In some applications we would like to have access to the subsets of models, so the framework is made flexible enough to also fit something like the category of subsets of models. We provide the motivating examples for AECats, arising from: first-order logic, positive logic, continuous logic and AECs.

AECats do not have syntax, but we can still make sense of a notion of types through the idea of Galois types, as we do in section \ref{sec:galois-types}. Since we do not have access to single elements in our category, we instead consider tuples of arrows, keeping in mind that each arrow can actually represent an entire tuple of elements. From this perspective, there is no difference between the domain of a type and its parameters.

An interesting property for Galois types is being finitely short, which says that the Galois type of a tuple is determined by the Galois types of its finite subtuples (\thref{def:finitely-short}). We do not need this property in the rest of this paper, but we mention it and the links it provides to existing frameworks in section \ref{sec:finitely-short-aecats}.

In section \ref{sec:isi-sequences-and-isi-dividing} we introduce the notion of isi-dividing for Galois types. We also discuss its connections to the usual notion of dividing.

In section \ref{sec:independence-relations} we introduce the notion of an independence relation as a relation on triples of subobjects. We formulate the properties it can have, and prove some basic facts about these properties, including how to derive \textsc{3-amalgamation} from a few other properties. This allows us to later compare simple and stable independence relations.

Finally, section \ref{sec:kim-pillay} contains the proof of the main theorem.

\vspace{\baselineskip}\noindent
\textbf{Acknowledgements.} I would like to thank my supervisor, Jonathan Kirby, for his invaluable input and feedback. I would also like to thank Marcos Mazari-Armida and Sebastien Vasey for their extensive feedback, and the anonymous referees whose remarks improved the presentation of this paper and provided a stronger connection with existing frameworks. This paper is part of a PhD project at the UEA (University of East Anglia), and as such is supported by a scholarship from the UEA.
\section{AECats}
\label{sec:aecats}
\begin{convention}
\thlabel{conv:regular-cardinals}
Throughout, $\kappa$, $\lambda$ and $\mu$ will denote regular infinite cardinals.
\end{convention}
Our framework is based on the category of models of some theory $T$, and the category of subsets of models of $T$.
\begin{definition}
\thlabel{def:category-of-models}
Given a first-order theory $T$, we denote by $\Mod(T)$ its category of models with elementary embeddings. We denote by $\SubMod(T)$ the category of subsets of models of $T$. That is, its objects are pairs $(A, M)$ where $A \subseteq M$ and $M$ is a model of $T$. An arrow $f: (A, M) \to (B, N)$ is then an elementary map $f: A \to B$. That is: for all $\bar{a} \in A$ and every formula $\phi(\bar{x})$ we have $M \models \phi(\bar{a})$ if and only if $N \models \phi(f(\bar{a}))$.
\end{definition}
There is a full and faithful embedding $\Mod(T) \hookrightarrow \SubMod(T)$, by sending $M$ to $(M, M)$. So we consider $\Mod(T)$ as a full subcategory of $\SubMod(T)$.

Due to the downward L\"owenheim-Skolem theorem, every model can be written as a union of models of cardinality at most $|T|$. This motivates the definition of an accessible category (see \cite{adamek_locally_1994} for an extensive treatment).
\begin{definition}
\thlabel{def:accessible-category}
A category $\C$ is called \emph{$\lambda$-accessible} if:
\begin{enumerate}[label=(\roman*)]
\item $\C$ has $\lambda$-directed colimits,
\item there is a set $\A$ of $\lambda$-presentable objects, such that every object in $\C$ can be written as a $\lambda$-directed colimit of objects in $\A$.
\end{enumerate}
A category is called \emph{accessible} if it is $\lambda$-accessible for some $\lambda$.
\end{definition}
We recall that an object $X$ is $\lambda$-presentable when $\Hom(X, -)$ preserves all $\lambda$-directed colimits. This means that if $Y = \colim_{i \in I} Y_i$ for some $\lambda$-directed diagram $(Y_i)_{i \in I}$ then every arrow $X \to Y$ factors essentially uniquely as $X \to Y_i \to Y$ for some $i \in  I$.

This gives us a notion of size. For example, in $\Mod(T)$ we have for $\lambda > |T|$ that $M$ is $\lambda$-presentable precisely when $|M| < \lambda$. Similarly, in $\SubMod(T)$, for any $\lambda$, an object $(A, M)$ is $\lambda$-presentable precisely when $|A| < \lambda$.

It is well-known that $\Mod(T)$ has directed colimits. In $\SubMod(T)$ directed colimits also exist: they are calculated by taking unions (in a big enough model). Thus $\Mod(T)$ and $\SubMod(T)$ are examples of accessible categories. Besides the existence of directed colimits (instead of just $\lambda$-directed colimits), these categories enjoy some other useful properties. For example, all arrows are monomorphisms and they have the amalgamation property.
\begin{definition}
\thlabel{def:amalgamation-base}
We say that a category has the \emph{amalgamation property} (or \emph{AP}) if given any span $N_1 \xleftarrow{f_1} M \xrightarrow{f_2} N_2$, there is a cospan $N_1 \xrightarrow{g_1} U \xleftarrow{g_2} N_2$, called an \emph{amalgam}, such that the following square commutes:
\[
\begin{tikzcd}
                      & U                                      &                        \\
N_1 \arrow[ru, "g_1"] &                                        & N_2 \arrow[lu, "g_2"'] \\
                      & M \arrow[ru, "f_2"'] \arrow[lu, "f_1"] &                       
\end{tikzcd}
\]
\end{definition}
The point of considering $\SubMod(T)$ is that we can later apply our results to arbitrary subsets of models. However, we do need to keep track of which objects are models.
\begin{definition}
\thlabel{def:aecat}
An \emph{AECat}, short for \emph{abstract elementary category}, consists of a pair $(\C, \M)$ where $\C$ and $\M$ are accessible categories and $\M$ is a full subcategory of $\C$ such that:
\begin{enumerate}[label=(\arabic*)]
\item $\M$ has directed colimits, which the inclusion functor into $\C$ preserves;
\item all arrows in $\C$ (and thus in $\M$) are monomorphisms.
\end{enumerate}
The objects in $\M$ are called \emph{models}. We say that $(\C, \M)$ has the \emph{amalgamation property} (or \emph{AP}) if $\M$ has the amalgamation property.
\end{definition}
The name ``abstract elementary category'' was used before in \cite[Definition 5.3]{beke_abstract_2012} for a very similar concept. As noted there as well, the name was used even before that in an unpublished note by Jonathan Kirby \cite{kirby_abstract_2008}.

Note that if $(\C, \M)$ is an AECat then $(\M, \M)$ is an AECat as well.
\begin{example}
\thlabel{ex:category-of-models-aecat}
As seen in the discussion before, both $(\Mod(T), \Mod(T))$ and $(\SubMod(T), \Mod(T))$ are AECats with AP. These are the prototypical examples of AECats to keep in mind.
\end{example}
To help with intuition that objects in $\C$ play the role of subsets of models, the reader may assume that for every object $A$ in $\C$, there is an arrow $A \to M$ with $M$ in $\M$. This is in fact true in all examples we consider and any object in $\C$ we will consider in this paper will always come with an arrow into some model anyway.
\begin{remark}
\thlabel{rem:chains-vs-directed-colimits}
Recall that a chain is a diagram of shape $\delta$, where $\delta$ is some ordinal. By \cite[Corollary 1.7]{adamek_locally_1994} we could equivalently replace ``directed colimits'' by ``colimits of chains'' in (1) in \thref{def:aecat}.
\end{remark}
\begin{remark}
\thlabel{rem:category-with-models-accessible-inclusion-functor}
If $(\C, \M)$ is an AECat then $\C$ and $\M$ may be accessible for different cardinals. By \cite[Corollary 2.14]{adamek_locally_1994} and \cite[Theorem 2.19]{adamek_locally_1994} there are arbitrarily large $\lambda$ such that both $\C$ and $\M$ are $\lambda$-accessible and the inclusion $\M \hookrightarrow \C$ preserves $\lambda$-presentable objects.
\end{remark}
Other applications of AECats include positive logic, continuous logic, quasiminimal excellent classes, AECs and compact abstract theories. We discuss those in the following examples.
\begin{example}
\thlabel{ex:category-of-models-positive-theory}
For an introduction to positive logic, we refer to \cite{poizat_positive_2018} or \cite{ben-yaacov_positive_2003}. The terminology in the latter differs significantly from the former, and we use the terminology of \cite{poizat_positive_2018} to recall the basics of positive logic. All claims in this example can be found there. The formulas of interest are the \emph{positive existential formulas}, these are of the form $\exists \bar{x} \phi(\bar{x}, \bar{y})$ where $\phi(\bar{x}, \bar{y})$ is \emph{positive quantifier-free} (i.e. built from atomic formulas using conjunction, disjunction, $\top$ and $\bot$). An \emph{h-inductive theory} $T$ is then a set of \emph{h-inductive sentences}. That is, sentences of the form $\forall \bar{x}(\phi(\bar{x}) \to \psi(\bar{x}))$ where $\phi(\bar{x})$ and $\psi(\bar{x})$ are positive existential. A \emph{homomorphism} of structures is a function that preserves truth of positive existential formulas, and it is called an \emph{immersion} if it also reflects truth of such formulas. So immersions are in particular injective. A model $M$ of $T$ is called \emph{positively closed} if every homomorphism from $M$ into another model of $T$ is an immersion.

Every first-order theory can be seen as an h-inductive theory through a process called \emph{positive Morleyisation}: for each formula $\phi(\bar{x})$ we introduce a relation symbol $R_\phi(\bar{x})$, and add axioms expressing $\forall \bar{x}(\neg \phi(\bar{x}) \leftrightarrow R_\phi(\bar{x}))$. Homomorphisms between models of such a theory will be elementary embeddings, and thus immersions. So every model is positively closed. Even though we expand the language in this process, it is clear that the category of models is not changed. So first-order logic can be studied as a special case of positive logic, and we will use the same notation.

For an h-inductive theory $T$, we define $\Mod(T)$ to be the category of positively closed models of $T$, with homomorphisms (and thus immersions) as arrows. For $\SubMod(T)$ we take as objects pairs $(A, M)$, where $A \subseteq M$ and $M$ is a positively closed model of $T$. An arrow $f: (A, M) \to (B, N)$ is then a function $f: A \to B$ that is an immersion on those sets. That is, for all $\bar{a} \in A$ and all positive existential $\phi(\bar{x})$, we have
\[
M \models \phi(\bar{a}) \quad \Longleftrightarrow \quad N \models \phi(f(\bar{a})).
\]
One easily checks that both these categories have directed colimits, which are calculated by taking the union in the usual way. The presentability of objects and accessibility for these categories is the same as in the first-order case. So we again get $(\Mod(T), \Mod(T))$ and $(\SubMod(T), \Mod(T))$ as AECats.

We have enough compactness in positive logic to prove the amalgamation property. The proof is similar to the first-order case. In fact, the essence of the argument for positive logic appears in \cite[Lemma 1.37]{ben-yaacov_positive_2003}, when combined with the method of diagrams.
\end{example}
\begin{example}
\thlabel{ex:continuous-logic}
In this example we consider continuous logic (see \cite{ben-yaacov_model_2008}). Given a continuous theory $T$, we can consider its category of models $\MetMod(T)$ with elementary embeddings. We use different notation to emphasise the continuous setting (even though we can encode first-order theories as continuous theories). We can also again consider subsets of such models, so we define $\SubMetMod(T)$ to be pairs $(A, M)$ where $M$ is a model of $T$ and $A \subseteq M$. An arrow $f: (A, M) \to (B, N)$ will be what is called an \emph{elementary map} in \cite[Definition 4.3(3)]{ben-yaacov_model_2008}.

The right notion of size is that of \emph{density character}: the smallest cardinality of a dense subset of the space. Denote the density character of a space $X$ by $\density(X)$. We then have for all $\lambda$ that $(A, M)$ in $\SubMetMod(T)$ is $\lambda$-presentable precisely when $\density(A) < \lambda$. For $\MetMod(T)$ we have that $M$ is $\lambda^+$-presentable precisely when $\density(M) < \lambda^+$, for all $\lambda$ such that the signature of $T$ has at most cardinality $\lambda$. As before, $(\MetMod(T), \MetMod(T))$ and $(\SubMetMod(T), \MetMod(T))$ form AECats with AP.

Checking all the properties is straightforward but lengthy. The reason they hold is due to the same tools (for which there exist a continuous alternative): directed colimits, L\"owenheim-Skolem and compactness.
\end{example}
\begin{example}
\thlabel{ex:aec}
Shelah's AECs are in particular also AECats. That is, given an AEC $\K$, we can view it as a category by taking as arrows the $\K$-embeddings: maps $f: M \to N$ such that $f(M) \preceq_\K N$ and $f$ is an isomorphism from $M$ onto $f(M)$. The Tarski-Vaught chain axioms are saying precisely that $\K$ has colimits of chains (and hence directed colimits, see \thref{rem:chains-vs-directed-colimits}). The L\"owenheim-Skolem axiom then guarantees that $\K$ is accessible. By definition every arrow in $\K$ is a monomorphism, so $(\K, \K)$ is an AECat. Of course, an AEC with AP will then be an AECat with AP.

We can generalise the construction of $\SubMod(T)$ to AECs. Let $\K$ be an AEC with AP. We  define a category $\SubSet(\K)$ as follows. The objects are pairs $(A, M)$ where $A \subseteq M$ and $M \in \K$. An arrow $f: (A, M) \to (B, N)$ is then a map $f: A \to B$ such that there are $\K$-embeddings $g: M \to U$ and $h: N \to U$ with $U \in \K$, making the following diagram commute:
\[
\begin{tikzcd}
                                   & U &                    \\
M \arrow[ru, "g"]                  &   & N \arrow[lu, "h"'] \\
A \arrow[rr, "f"'] \arrow[u, hook] &   & B \arrow[u, hook] 
\end{tikzcd}
\]
The amalgamation property is needed to compose arrows. Now suppose that our AEC is fully $<\!\lambda$-type short over the empty set, as defined in \cite[Definition 3.3]{boney_tameness_2014}. Let $(A_i, M_i)_{i \in I}$ be a $\lambda$-directed diagram in $\SubSet(\K)$ and suppose that the obvious cocone $(\bigcup_{i \in I} A_i, M)$ exists for some $M \in \K$. Then $(\bigcup_{i \in I} A_i, M)$ is the colimit of $(A_i, M_i)_{i \in I}$, where the type shortness is necessary to prove the universal property. Under these conditions $\SubSet(\K)$ has $\lambda$-directed colimits. It is then straightforward to verify that for $\kappa \geq \lambda$ an object $(A, M)$ of $\SubSet(\K)$ is $\kappa$-presentable precisely when $|A| < \kappa$. It follows that for $\kappa \geq \lambda + \operatorname{LS}(\K)^+$, we have that $\SubSet(\K)$ is $\kappa$-accessible. So $(\SubSet(\K), \K)$ is an AECat with AP.

Note that for $\lambda = \aleph_0$ we have automatically that the cocone $(\bigcup_{i \in I} A_i, M)$ exists. This can be seen by following the proofs in for example \cite[Remark 2.34]{ben-yaacov_positive_2003} or \cite[Lemma 1.3]{buechler_simple_2003}.
\end{example}
\begin{example}
\thlabel{ex:cat}
In \cite{ben-yaacov_positive_2003} the concept of a \emph{compact abstract theory}, or \emph{cat}, is introduced. Although no formal definition is given, it turns out that in practice such a cat is in fact an AECat with AP. See also \thref{ex:elementary-category}.
\end{example}
\begin{example}
\thlabel{ex:quasiminimal-excellent-classes}
In this example we consider Zilber's quasiminimal excellent classes. We use the terminology from \cite{kirby_quasiminimal_2010}. Let $\C$ be a quasiminimal excellent class, also satisfying axiom IV. Then $\C$ together with strong embeddings is a finitely accessible category, where $M \in \C$ is $\kappa$-presentable precisely when $M$ has dimension $< \kappa$. So $(\C, \C)$ is an AECat with AP.
\end{example}
We have now covered how existing frameworks can be placed in the framework of AECats. In section \ref{sec:finitely-short-aecats} we will do the converse. There we discuss how, under some additional assumptions, AECats can be placed in existing frameworks.

\section{Galois types}
\label{sec:galois-types}
In \cite[Definition II.1.9]{shelah_universal_1987} types are considered as the orbit of a tuple under some automorphism group. Later this idea was generalised by replacing the automorphisms by embeddings into a bigger model, and the name Galois type was introduced (see \cite{grossberg_classification_2002}). We use this idea, replacing elements by arrows.
\begin{definition}
\thlabel{def:extension}
Let $M$ be a model in an AECat. An \emph{extension} of $M$ is an arrow $M \to N$, where $N$ is some model.
\end{definition}
\begin{convention}
\thlabel{conv:extending-monomorphisms}
Usually, there will be only one relevant extension of models. So to prevent cluttering of notation we will not give such an extension a name. Given such an extension $M \to N$ and some arrow $a: A \to M$ we will then denote the arrow $A \xrightarrow{a} M \to N$ by $a$ as well.
\end{convention}
\begin{definition}
\thlabel{def:galois-type}
Let $(\C, \M)$ be an AECat with AP. We will use the notation $((a_i)_{i \in I}; M)$ to mean that the $a_i$ are arrows into $M$ and that $M$ is a model. We will denote the domain of $a_i$ by $A_i$, unless specified otherwise.

We say that two tuples $((a_i)_{i \in I}; M)$ and $((a_i')_{i \in I}; M')$ have the same \emph{Galois type}, and write
\[
\gtp((a_i)_{i \in I}; M) = ((a_i')_{i \in I}; M'),
\]
if $\dom(a_i) = \dom(a_i')$ for all $i \in I$, and there is a common extension $M \to N \leftarrow M'$, such that, for all $i \in I$, $a_i$ and $a_i'$ give the same arrow into $N$. That is, the following commutes for all $i \in I$:
\[
\begin{tikzcd}
 & N &  \\
M \arrow[ru] &  & M' \arrow[lu] \\
 & A_i \arrow[lu, "a_i"] \arrow[ru, "a_i'"'] & 
\end{tikzcd}
\]
\end{definition}
Note that AP ensures that having the same Galois type is an equivalence relation. For this reason, we will only be interested in AECats with AP in the rest of this paper.
\begin{fact}
\thlabel{fact:galois-type-invariance-under-extension}
Let $M \to N$ be any extension, then for any tuple $((a_i)_{i \in I}; M)$:
\[
\gtp((a_i)_{i \in I}; M) = \gtp((a_i)_{i \in I}; N).
\]
\end{fact}
This is a good example of \thref{conv:extending-monomorphisms}. A more precise statement would be to give the extension $M \to N$ a name, say $f$, then for any $((a_i)_{i \in I}; M)$ we have that $\gtp((a_i)_{i \in I}; M) = \gtp((fa_i)_{i \in I}; N)$.

Since all arrows in an AECat are monomorphisms, every arrow will represent a subobject. Later, in section \ref{sec:independence-relations}, we will work a lot with subobjects. So we extend our notation to subobjects.
\begin{definition}
\thlabel{def:galois-type-subobjects}
Let $((A_i)_{i \in I}; M)$ and $((A_i')_{i \in I}; M')$ be two tuples of subobjects in an AECat with AP. Then we say that they have the same \emph{Galois type}, denoted as
\[
\gtp((A_i)_{i \in I}; M) = \gtp((A_i')_{i \in I}; M'),
\]
if there are extensions $M \to N \leftarrow M'$ such that $A_i = A_i'$ as subobjects of $N$ (we consider them subobjects of $N$ by composing with the relevant extension).
\end{definition}
\begin{proposition}
\thlabel{prop:galois-type-subobjects-representatives}
Let $((A_i)_{i \in I}; M)$ and $((A_i')_{i \in I}; M')$ be two tuples of subobjects. Then
\[
\gtp((A_i)_{i \in I}; M) = \gtp((A_i')_{i \in I}; M')
\]
precisely if given any representatives $(a_i)_{i \in I}$ of $(A_i)_{i \in I}$ there are representatives $(a_i')_{i \in I}$ of $(A_i')_{i \in I}$, such that
\[
\gtp((a_i)_{i \in I}; M) = \gtp((a_i')_{i \in I}; M').
\]
\end{proposition}
\begin{proof}
The right to left direction is trivial. For the other direction, we let $M \to N \leftarrow M'$ be such that $A_i = A_i'$ as subobjects of $N$ for all $i \in I$. Let representatives $(a_i)_{i \in I}$ of $(A_i)_{i \in I}$ be given and pick some representatives $(b_i: A_i' \to M')_{i \in I}$ of $(A_i')_{i \in I}$. Because $A_i = A_i'$ as subobjects of $N$, there must be an isomorphism $f_i: A_i \to A_i'$ for each $i \in I$, making the following diagram commute:
\[
\begin{tikzcd}
                                       & N &                         \\
M \arrow[ru]                           &   & M' \arrow[lu]           \\
A_i \arrow[u, "a_i"] \arrow[rr, "f_i"] &   & A_i' \arrow[u, "b_i"']
\end{tikzcd}
\]
In particular, $b_i f_i$ also represents $A_i'$ as a subobject of $M'$. So we can take $a_i' = b_i f_i$ for all $i \in I$. By construction we then have that $\gtp((a_i)_{i \in I}; M) = \gtp((a_i')_{i \in I}; M')$.
\end{proof}
The following example illustrates why we have to be careful when moving to representatives of subobjects.
\begin{example}
\thlabel{ex:galois-types-on-subobjects}
Consider the category of infinite sets with injective functions. This is easily seen to be an AECat with AP if we take $\M$ to be the entire category. Alternatively, this is precisely $\Mod(T_\text{inf})$, where $T_\text{inf}$ is the theory of infinite sets, and is thus an AECat with AP as discussed in \thref{ex:category-of-models-aecat}.

Let $f: \N \to \N$ be the bijection that swaps the odd and even numbers. So $f(0) = 1$, $f(1) = 0$, $f(2) = 3$, and so on. Denote by $2\N$ the set of even numbers and let $e: 2\N \to \N$ be the inclusion. So we have the following commuting diagram:
\[
\begin{tikzcd}
                                                               & \N &                \\
\N \arrow[ru, "f"', bend right] \arrow[ru, "Id_\N", bend left]\arrow[loop left, "f"] &    & 2\N \arrow[lu, "e"']
\end{tikzcd}
\]
We denote by $[Id_\N]$ the subobject represented by $Id_\N$, and likewise for $f$ and $e$. Then $[Id_\N] = [f]$, so we definitely have
\[
\gtp([Id_\N], [e]; \N) = \gtp([f], [e]; \N).
\]
However, we cannot have
\[
\gtp(Id_\N, e; \N) = \gtp(f, e; \N).
\]
So we cannot just pick any representatives of the subobjects.
\end{example}
The intuition here is that a type cares about the way a certain set is enumerated, while a subobject only cares about the set itself. So different enumerations of a certain set may yield incompatible types, while they represent the same subobject.
\begin{proposition}
\thlabel{prop:galois-types-basic-observations}
Suppose we have $\gtp((a_i)_{i \in I}; M) = \gtp((a_i')_{i \in I}; M')$, then:
\begin{enumerate}[label=(\roman*)]
\item for any $I_0 \subseteq I$ we have $\gtp((a_i)_{i \in I_0}; M) = \gtp((a_i')_{i \in I_0}; M')$;
\item suppose that we also have an arrow $b_i: B_i \to A_i$ for each $i \in I$, then
\[
\gtp((a_i)_{i \in I}, (a_i b_i)_{i \in I}; M) = \gtp((a_i')_{i \in I}, (a_i' b_i)_{i \in I}; M')
\]
and thus $\gtp((a_i b_i)_{i \in I}; M) = \gtp((a_i' b_i)_{i \in I}; M')$;
\item let $b: B \to M$ be some arrow, then there is an extension $M' \to N$ and some $b': B \to N$ such that $\gtp(b, (a_i)_{i \in I}; M) = \gtp(b', (a_i')_{i \in I}; N)$.
\end{enumerate}
\end{proposition}
\begin{proof}
For (i) and (ii) the common extension witnessing the original equality will also witness the new equality. The last sentence from (ii) follows from applying (i).

For (iii) let $M \xrightarrow{f} N \xleftarrow{g} M'$ be witnesses of $\gtp((a_i)_{i \in I}; M) = \gtp((a_i')_{i \in I}; M')$. We define $b' = fb$, so that:
\[
\gtp(b, (a_i)_{i \in I}; M) =
\gtp(fb, (fa_i)_{i \in I}; N) =
\gtp(b', (ga_i')_{i \in I}; N).
\]
Then the result follows directly if we take the extension $M' \to N$ to be $g$, so that we would write the right-hand side as $\gtp(b', (a_i')_{i \in I}; N)$.
\end{proof}
\begin{proposition}
\thlabel{prop:factorisation-captured-by-galois-type}
Suppose we have $(a, b; M)$, such that $a = bi$ for some arrow $i$. If then $(a', b'; M')$ is such that
\[
\gtp(a, b; M) = \gtp(a', b'; M'),
\]
then $a'$ factors through $b'$ in the same way: $a' = b'i$.
\end{proposition}
\begin{proof}
From $\gtp(a, b; M) = \gtp(a', b'; M')$ we get extensions $M \to N \leftarrow M'$ and a diagram
\[
\begin{tikzcd}
                  & N                                                                         &                     \\
M \arrow[ru, "f"] & B \arrow[l, "b"] \arrow[r, "b'"']                                         & M' \arrow[lu, "g"'] \\
                  & A \arrow[lu, "a", bend left] \arrow[u, "i"] \arrow[ru, "a'"', bend right] &                    
\end{tikzcd}
\]
where everything commutes by definition except for possibly the bottom right triangle (i.e. the triangle $a' = b'i$). So we have $ga' = fa = fbi = gb'i$ and so $a' = b'i$ because $g$ is a monomorphism.
\end{proof}
\begin{remark}
\thlabel{rem:monster-objects}
It is standard in model theory to work with monster models. In the general category-theoretic setting this would still be possible. For example, in \cite{lieberman_classification_2014} it is shown that such monster objects exist in any accessible category with directed colimits and the amalgamation property, assuming some additional set theory. We choose not to work with monster objects. This might come at some notational cost, but it keeps everything within the standard set theory.
\end{remark}
\section{Finitely short AECats}
\label{sec:finitely-short-aecats}
In this section we discuss an important property that connects AECats with existing frameworks. Nothing in this section will be used in the rest of this paper.

This property is to have some locality for Galois types (inspired by \cite{grossberg_categoricity_2006}): a Galois type of an infinite tuple should be determined by all its finite subtuples. This can even be used to get some compactness for directed systems of Galois types (see e.g. \cite[Remark 2.34]{ben-yaacov_positive_2003} and \cite[Lemma 1.3]{buechler_simple_2003}).
\begin{definition}
\thlabel{def:finitely-short}
We say that an AECat is \emph{finitely short} if for any two (infinite) tuples $((a_i)_{i \in I}; M)$ and $((a_i')_{i \in I}; M')$ such that for all finite $I_0 \subseteq I$
\[
\gtp((a_i)_{i \in I_0}; M) = \gtp((a_i')_{i \in I_0}; M'),
\]
we have that
\[
\gtp((a_i)_{i \in I}; M) = \gtp((a_i')_{i \in I}; M').
\]
\end{definition}
\begin{example}
\thlabel{ex:tameness-in-category-of-models}
The AECats $(\Mod(T), \Mod(T))$ and $(\SubMod(T), \Mod(T))$ from \thref{ex:category-of-models-positive-theory} are both finitely short (recall that this includes the first-order case), because Galois types coincide with the usual syntactic types.

For the same reasons, for a continuous theory $T$, $(\MetMod(T), \MetMod(T))$ and $(\SubMetMod(T), \MetMod(T))$ from \thref{ex:continuous-logic} are finitely short.

An AEC $\K$ with AP that is fully $<\!\aleph_0$-type short over the empty set yields AECats $(\K, \K)$ and $(\SubSet(\K), \K)$, as per \thref{ex:aec}, which are both finitely short.
\end{example}
\begin{example}
\thlabel{ex:elementary-category}
In \thref{ex:cat} we mentioned cats from \cite{ben-yaacov_positive_2003}. One definition there allows for a nice comparison to AECats, namely that of an \emph{elementary category (with amalgamation)} \cite[Definition 2.27]{ben-yaacov_positive_2003}. This is a concrete category $\C$ that satisfies a few additional assumptions, similar to the axioms of an AEC. Every such elementary category $\C$ will form an AECat with AP as $(\C, \C)$, if we additionally assume $\C$ to be accessible\footnote{Technically, \cite[Definition 2.27]{ben-yaacov_positive_2003} does not require the existence of directed colimits but something slightly weaker called the ``elementary chain property''. However, it is likely that actually directed colimits are meant and in practice this is what we have.}.

Conversely, given an AECat $(\C, \M)$ we can make it into a concrete category using a version of the Yoneda embedding. Let $\lambda$ be such that $\C$ is $\lambda$-accessible and let $\A$ be the full subcategory of $\lambda$-presentable objects in $\C$. Then there is a fully faithful canonical functor $E: \C \to \Set^{\A^\op}$ that preserves $\lambda$-directed colimits, see \cite[1.25 and 2.8]{adamek_locally_1994}. If $(\C, \M)$ has AP then taking the image of $\M$ under $E$, we obtain an elementary category with amalgamation.

In \cite[Definition 2.32]{ben-yaacov_positive_2003} a few properties are defined for the Galois types:
\begin{itemize}
\item \emph{type boundedness}: this is always true in an AECat, see \thref{prop:set-of-galois-types};
\item \emph{type locality}: this is precisely what we called being finitely short;
\item \emph{weak compactness}: this holds for example in categories obtained from a first-order, positive or continuous theory.
\end{itemize}
\end{example}
Note that \thref{ex:elementary-category} does generally not yield an AEC. For example, take $\C = \M$ to be the category of infinite sets with injective functions. If we make this category into a concrete category through the functor $E: \C \to \Set^{\A^\op}$ then $E(\omega + \omega)$ contains the arrow $f: \omega \to \omega + \omega$ where $f(n) = \omega + n$. If this would be an AEC then $E(\omega + \omega) = \bigcup_{n < \omega} E(\omega + n)$, but $f$ is not in $E(\omega + n)$ for any $n$. So the Tarski-Vaught chain axiom for AECs fails. The point is of course that a directed colimit can be more than just the union of underlying sets.

In \cite[Corollary 5.7]{beke_abstract_2012} a characterisation is given of those accessible categories that are equivalent to an AEC. It also describes how to construct an AEC from such an accessible category, through a construction very similar to \thref{ex:elementary-category}.
\begin{example}
\thlabel{ex:homogeneous-model-theory}
Let $(\C, \M)$ be a finitely short AECat with AP. Suppose furthermore that $\M$ has the \emph{joint embedding property}. That is, for any two models $M_1$ and $M_2$ there is a third model $N$ with arrows $M_1 \to N \leftarrow M_2$. Then there is a strong connection with homogeneous model theory \cite{buechler_simple_2003}. We sketch the construction and would like to thank an anonymous referee for pointing this out.

As discussed in \thref{ex:elementary-category}, we can turn $\M$ into a concrete category. So using the usual tools we can build a monster model $\MM$, which we will fit in the framework of \cite{buechler_simple_2003}. The elements in $\MM$ are arrows in $\C$, and having the same Galois type corresponds to having the same orbital type in $\MM$. For every Galois type of a finite tuple we add a relation symbol, and we close these under finite conjunctions and disjunctions. Then two (infinite) tuples of elements are in the same orbit of $\MM$ iff they have the same Galois type iff their finite subtuples have the same Galois type iff the finite subtuples satisfy the same relation symbols.
\end{example}
The constructions in \thref{ex:elementary-category} and \thref{ex:homogeneous-model-theory} do not change our category, they only add data to make it into a concrete category. So any notion that is defined on just the objects and arrows in our category is preserved by this operation. In particular independence relations, as we define in section \ref{sec:independence-relations}, are preserved. It would be interesting to study these connections further, but that is beyond the scope of this paper.
\begin{definition}
\thlabel{def:galois-type-set}
Let $(\C, \M)$ be an AECat with AP. For a tuple $(A_i)_{i \in I}$ of objects in $\C$, let $S((A_i)_{i \in I})$ be the collection of all tuples $((a_i)_{i \in I}; M)$ such that $\dom(a_i) = A_i$. We define the \emph{Galois type set} $\Sgtp((A_i)_{i \in I})$ as:
\[
\Sgtp((A_i)_{i \in I}) = S((A_i)_{i \in I}) / \sim_{gtp},
\]
where $\sim_{gtp}$ is the equivalence relation of having the same Galois type.
\end{definition}
An AECat is generally a large category. So $S((A_i)_{i \in I})$ will generally be a proper class. Below we prove that $\Sgtp((A_i)_{i \in I})$ is small (i.e. a set), so the name is justified.

The above notation clashes with standard notation where one would expect $(A_i)_{i \in I}$ to denote the parameters of the types. However, for our Galois types the difference between domain and parameters fades.

\thref{def:galois-type-set} allows us to talk about $\gtp((a_i)_{i \in I}; M)$ as an object in itself: it is one of the equivalence classes in $\Sgtp((A_i)_{i \in I})$.
\begin{proposition}
\thlabel{prop:set-of-galois-types}
Let $(\C, \M)$ be an AECat with AP. Then for any tuple $(A_i)_{i \in I}$ of objects $\Sgtp((A_i)_{i \in I})$ is a set.
\end{proposition}
\begin{proof}
We prove that there is a subset $S'((A_i)_{i \in I}) \subseteq S((A_i)_{i \in I})$, such that for every tuple $((a_i)_{i \in I}; M) \in S((A_i)_{i \in I})$, there is some $((a_i')_{i \in I}; M') \in S'((A_i)_{i \in I})$ with $\gtp((a_i)_{i \in I}; M) = \gtp((a_i')_{i \in I}; M')$.

Let $\lambda$ be such that every $A_i$ is $\lambda$-presentable, $\lambda > |I|$ and the inclusion functor $\M \hookrightarrow \C$ is $\lambda$-accessible and preserves $\lambda$-accessible objects. Such a $\lambda$ must exist since each object in an accessible category is presentable by \cite[Proposition 1.16]{adamek_locally_1994}, and by \thref{rem:category-with-models-accessible-inclusion-functor}.

Let $\M_\lambda$ be (a skeleton of) all the models that are $\lambda$-presentable. Then $\M_\lambda$ is a set (see the remark after \cite[Definition 1.9]{adamek_locally_1994}). For an object $M$. We define:
\[
S'((A_i)_{i \in I}) = \coprod_{M \in \M_\lambda} \prod_{i \in I} \Hom(A_i, M)
\]
We check that $S'((A_i)_{i \in I})$ has the required property. Let $((a_i)_{i \in I}; M) \in S((A_i)_{i \in I})$. Then because $\M$ is $\lambda$-accessible, $M$ is a $\lambda$-directed colimit of $\lambda$-presentable objects $(M_j)_{j \in J}$. That is, objects in $\M_\lambda$. Since the inclusion functor $\M \hookrightarrow \C$ preserves directed colimits, we still have $M = \colim_{j \in J} M_j$ in $\C$. As $A_i$ is $\lambda$-presentable for each $i \in I$, we have that each $a_i$ factors through some $M_{j_i}$. Then since $\lambda > |I|$, there is $j \in J$ such that every $a_i$ factors through $M_j$. Write this factorisation as $A_i \xrightarrow{a_i'} M_j \xrightarrow{m_j} M$, where $m_j$ is the coprojection from the colimit. Then by construction $((a_i')_{i \in I}; M_j) \in S'((A_i)_{i \in I})$ and $\gtp((a_i)_{i \in I}; M) = \gtp((a_i')_{i \in I}; M_j)$.
\end{proof}
\section{Isi-sequences and isi-dividing}
\label{sec:isi-sequences-and-isi-dividing}
\begin{definition}
\thlabel{def:sequence}
A \emph{sequence} is a tuple $((a_i)_{i \in I}; M)$ where every $a_i$ has the same domain and $I$ is a linear order.
\end{definition}
We will often need to treat an initial segment of a sequence as one object. The following definition makes sense of this in a category-theoretic setting.
\begin{definition}
\thlabel{def:chain-of-initial-segments}
A chain $(M_i)_{i < \kappa}$ is called \emph{continuous} if for every limit $\ell < \kappa$ we have $M_\ell = \colim_{i < \ell} M_i$. A \emph{chain of initial segments} for a sequence $(a_i)_{i < \kappa}$ in some $M$ is a continuous chain $(M_i)_{i < \kappa}$ of models with chain bound $M$ (i.e. we have arrows $m_i: M_i \to M$ forming a cocone for $(M_i)_{i < \kappa}$). Such that for all $i <\kappa$ we have that $a_i$ factors through $M_{i+1}$.

If an arrow $c: C \to M$ factors through the chain $(M_i)_{i < \kappa}$, so $c$ factors as $C \to M_0 \to M$, then we say that \emph{$c$ embeds in $(M_i)_{i < \kappa}$}.
\end{definition}
\begin{convention}
\thlabel{conv:chain-of-initial-segments}
For a chain of initial segments $(M_i)_{i < \kappa}$ for some sequence $(a_i)_{i < \kappa}$ in $M$ we will abuse notation and view $a_i$ as an arrow into $M_j$ for $i < j$. Similarly, if $C$ embeds in $(M_i)_{i < \kappa}$, we view $c$ as an arrow into $M_i$ for all $i < \kappa$.
\end{convention}
\begin{definition}
\thlabel{def:isi-sequence}
We call a sequence $(a_i)_{i < \kappa}$ in $M$, together with a chain of initial segments $(M_i)_{i < \kappa}$, an \emph{isi-sequence} (short for \emph{initial segment invariant}) if for all $i \leq j < \kappa$ we have:
\[
\gtp(a_i, m_i; M) = \gtp(a_j, m_i; M).
\]
For $c: C \to M$ we say this is an \emph{isi-sequence over $c$} if $c$ embeds in $(M_i)_{i < \kappa}$.
\end{definition}
\begin{definition}
\thlabel{def:consistency}
Suppose we have $(a, b, c; M)$, a sequence $(b_i)_{i \in I}$ in $N$ and $(c; N)$. Then we say that $\gtp(a, b, c; M)$ is \emph{consistent} for $(b_i)_{i \in I}$ if there is an extension $N \to N'$ and an arrow $a': A \to N'$ such that
\[
\gtp(a, b, c; M) = \gtp(a', b_i, c; N')
\]
for all $i \in I$. We call $a'$ a \emph{realisation} of $\gtp(a, b, c; M)$ for $(b_i)_{i \in I}$.
\end{definition}
We overloaded the notation for the arrow $c$: it denotes both an arrow into $M$ and into $N$. We want to think of $c$ as some fixed set of parameters, and this notation supports that. The context should make clear which arrow is meant.
\begin{lemma}
\thlabel{lem:chain-consistency}
Suppose we have $(a, b, c; M)$, a sequence $(b_i)_{i < \kappa}$ in $N$ together with a chain of initial segments $(M_i)_{i < \kappa}$ and $(c; N)$ that embeds in $(M_i)_{i < \kappa}$. Then $\gtp(a, b, c; M)$ is consistent for $(b_i)_{i < \kappa}$ if and only if there is some chain bound $N'$ of $(M_i)_{i < \kappa}$ and an arrow $a': A \to N'$ such that for all $i < \kappa$:
\[
\gtp(a, b, c; M) = \gtp(a', b_i, c; N').
\]
\end{lemma}
\begin{proof}
The left to right direction is direct: an extension $N \to N'$ together with a realisation $a'$ will be the required chain bound and arrow.

For the converse, let $N'$ and $a'$ be as in the statement. Define $M' = \colim_{i < \kappa} M_i$. Then we get extensions $N \leftarrow M' \to N'$, because $N$ and $N'$ are chain bounds of $(M_i)_{i < \kappa}$. Since $M'$ is a colimit of of models, it is a model itself and hence an amalgamation base. We thus find an amalgam $N \to N^* \leftarrow N'$. Then $N \to N^*$ together with $A \xrightarrow{a'} N' \to N^*$ are the required extension and realisation.
\end{proof}
We introduce the notion of isi-dividing, which is essentially extracted from the proof of the Kim-Pillay theorem. The goal was to get around the use of compactness and type locality (which are discussed in section \ref{sec:finitely-short-aecats}).
\begin{definition}
\thlabel{def:isi-dividing}
We say that $\gtp(a, b, c; M)$ \emph{isi-divides} over $c$ if there is $\mu$ such that for every $\lambda \geq \mu$ there is an isi-sequence $(b_i)_{i < \lambda}$ over $c$ in some $N$ with $\gtp(b_i, c; N) = \gtp(b, c; M)$ for all $i < \lambda$. Such that for some $\kappa < \lambda$ and every $I \subseteq \lambda$ with $|I| \geq \kappa$ we have that $\gtp(a, b, c; M)$ is inconsistent for $(b_i)_{i \in I}$.
\end{definition}
\begin{remark}
\thlabel{rem:isi-dividing-vs-dividing}
In the setting of first-order logic, positive logic \cite{pillay_forking_2000, ben-yaacov_simplicity_2003} and homogeneous model theory \cite{buechler_simple_2003} we have the usual notion of dividing, which is defined using indiscernible sequences.

It is not hard to see that dividing of some type $p$ implies isi-dividing of $p$ in those cases. All these settings have enough compactness to find arbitrarily long indiscernible sequences that witness dividing. Such an indiscernible sequence can be turned into an isi-sequence by inductively constructing models over which the tail of the sequence is indiscernible. So we find an isi-sequence where $p$ is inconsistent for every infinite subsequence, hence $p$ isi-divides.

In particular, if isi-dividing has local character, then dividing has local character. So local character of isi-dividing still implies that a first-order theory is simple, and also for positive theories, in the sense of \cite{ben-yaacov_simplicity_2003}.
\end{remark}
The converse of \thref{rem:isi-dividing-vs-dividing} is not clear to us. There are some partial answers in \thref{rem:isi-dividing-partial-answers}.
\begin{question}
\thlabel{q:isi-dividing-implies-dividing}
In settings where dividing has been defined in terms of indiscernible sequences (first-order, positive logic, homogeneous model theory), does isi-dividing imply dividing?
\end{question}
\begin{remark}
\thlabel{rem:isi-dividing-partial-answers}
If there exists a proper class of Ramsey cardinals then \thref{q:isi-dividing-implies-dividing} can be answered positively. In that case isi-dividing of some $\tp(a/Cb)$ is always witnessed by an isi-sequence of length some Ramsey cardinal $\lambda > |b| + |C|$. So the isi-sequence contains an indiscernible subsequence of length $\lambda$, which then witnesses dividing.

In the first-order and positive logic setting, dividing in a simple theory will give what we call a simple independence relation (see \thref{def:simple-and-stable-independence-relation} and \cite{kim_simple_1997, pillay_forking_2000, ben-yaacov_simplicity_2003}). Then by \thref{thm:kim-pillay-category-theoretic} this must coincide with the independence relation given by isi-dividing. So in simple theories dividing and isi-divding will coincide.
\end{remark}
Whenever we have some compactness, arguments using Ramsey cardinals can often be emulated by the usual Erd\H{o}s-Rado argument (see e.g.  \cite[Lemma 1.2]{ben-yaacov_simplicity_2003} or \cite[Lemma 1.4]{buechler_simple_2003}). So it seems reasonable to expect a positive answer to \thref{q:isi-dividing-implies-dividing}. For settings lacking compactness, the following is a natural question.
\begin{question}
\thlabel{q:isi-dividing-basic-properties}
What basic properties does isi-dividing generally satisfy? Can we prove more properties assuming certain combinatorial properties (e.g. local character), like we can for dividing?
\end{question}
\section{Independence relations}
\label{sec:independence-relations}
We will define an independence relation as a ternary relation on subobjects. The idea is similar to \cite{lieberman_forking_2019}. We compare the two further in \thref{rem:lrv-independence}.

We recall that the collection of subobjects $\Sub(X)$ forms a poset in any (well-powered) category, and if $A \leq B$ for $A, B \in \Sub(X)$, then we may also consider $A$ to be a subobject of $B$, that is $A \in \Sub(B)$. On the other hand, we always have $X \in \Sub(X)$ as the maximal element of this poset. So we will use the notation $A \leq X$ to mean that $A$ is a subobject of $X$.
\begin{convention}
\thlabel{conv:extending-subobjects}
We extend \thref{conv:extending-monomorphisms} to subobjects: given an extension $M \to N$ and a subobject $A \leq M$, we will view $A$ as a subobject of $N$.
\end{convention}
\begin{definition}
\thlabel{def:independence-relation}
In an AECat with AP, an \emph{independence relation} is a relation on triples of subobjects of models. If such a triple $(A, B, C)$ of a model $M$ is in the relation, we call it \emph{independent} and denote this by:
\[
A \indep[C]{M} B.
\]
This notation should be read as ``$A$ is independent from $B$ over $C$ (in $M$)''.
\end{definition}
\begin{definition}
\thlabel{def:independence-relation-properties}
An independence relation can have the following properties.
\begin{description}
\item[\textsc{Invariance}] If $A \indep[C]{M} B$ and $\gtp(A, B, C; M) = \gtp(A', B', C'; M')$ then we also have $A' \indep[C']{M'} B'$.

\item[\textsc{Monotonicity}] $A \indep[C]{M} B$ and $B' \leq B$ implies $A \indep[C]{M} B'$.

\item[\textsc{Base-Monotonicity}] $A \indep[C]{M} B$ and $C \leq C' \leq B$ implies $A \indep[C']{M} B$.

\item[\textsc{Transitivity}] $A \indep[B]{M} C$ and $A \indep[C]{M} D$ with $B \leq C$ implies that there is an extension $M \to N$ and $E \leq N$ such that $A \indep[B]{N} E$ and $C, D \leq E$.

\item[\textsc{Symmetry}] $A \indep[C]{M} B$ implies $B \indep[C]{M} A$.

\item[\textsc{Existence}] For $(a, c; M)$ and $B \leq M$ there is an extension $M \to M'$ with some $a': A \to M'$ such that $A' \indep[C]{M'} B$ and $\gtp(a', c; M') = \gtp(a, c; M)$.

\item[\textsc{Union}] Let $(B_i)_{i \in I}$ be a directed system with a cocone into some model $M$, and suppose $B = \colim_{i \in I} B_i$ exists. Then if $A \indep[C]{M} B_i$ for all $i \in I$, we have $A \indep[C]{M} B$.

\item[\textsc{Stationarity}] Let $A, B$ be objects, $M$ a model, and suppose we have corresponding arrows $a$, $b$ and $m$ into some $N$ and similar arrows $a'$, $b'$, $m'$ into some $N'$, such that $\gtp(a, m; N) = \gtp(a', m'; N')$ and $\gtp(b, m; N) = \gtp(b', m'; N')$. Then $A \indep[M]{N} B$ and $A' \indep[M']{N'} B'$ implies $\gtp(a, b, m; N) = \gtp(a', b', m'; N')$.
\end{description}
\end{definition}
\begin{remark}
\thlabel{rem:independence-relation-properties}
A few remarks about \thref{def:independence-relation-properties}.
\begin{enumerate}
\item We will mainly be interested in independence relations satisfying \textsc{Symmetry}. So, for example, we can apply \textsc{Monotonicity} to both sides. That is, if $A \indep[C]{M} B$ and $A' \leq A$, then $A' \indep[C]{M} B$. If the independence relation does not have \textsc{Symmetry}, one would have to distinguish between ``left'' and ``right'' versions (e.g. \textsc{Left-Monotonicity} and \textsc{Right-Monotonicity}).

\item If we have \textsc{Invariance}, \textsc{Monotonicity} and \textsc{Transitivity} then from $A \indep[B]{M} C$ and $A \indep[C]{M} D$ we can also conclude $A \indep[B]{M} D$. Most uses of \textsc{Transitivity} will actually be of this form, and we will just refer to it as ``by \textsc{Transitivity}''.

\item The usual extension property is implied by \textsc{Invariance}, \textsc{Monotonicity}, \textsc{Transitivity} and \textsc{Existence}: see \thref{{prop:basic-properties-independence-relation}}(iii).

\item The \textsc{Union} property is our version of what is usually known as ``finite character''. In a concrete setting it follows directly from finite character, but this formulation is more suited for our category-theoretic setting. In the setting of AECs one often sees the name ``$(<\aleph_0)$-witness property'', which implies \textsc{Union}.

\item In the statement of \textsc{Union}: we can view $B$ as a subobject of $M$ because the universal property of the colimit guarantees an arrow $B \to M$, which must be a monomorphism because all arrows are monomorphisms in an AECat. If every $B_i$ is a model, then the colimit $B$ always exists and is a model. Throughout the paper, we will only need to apply \textsc{Union} to directed systems of models.

\item \textsc{Stationarity} is sometimes also called ``uniqueness''.
\end{enumerate}
\end{remark}
There are two more key properties: \textsc{Local Character} and \textsc{3-amalgamation}. The first one is usually defined on finite objects, but these may not exist in our category. So, similar to \cite[Definition 8.6]{lieberman_forking_2019}, we have to build in some dependence on the size of the objects involved.
\begin{definition}
\thlabel{def:local-character}
An independence relation has \textsc{Local Character} if for every cardinal $\lambda$, there is a cardinal $\Upsilon(\lambda)$, such that the following holds. Given a model $M$ with subobjects $A$ and $B$, where $A$ is $\lambda$-presentable, there is an $\Upsilon(\lambda)$-presentable $B' \leq B$ such that $A \indep[B']{M} B$. We call the class function $\Upsilon$ a \emph{local character function}.
\end{definition}
Note that we do not consider a local character function to be part of the data for an independence relation. Being a local character function is just saying that it witnesses the \textsc{Local Character} property.
\begin{convention}
\thlabel{conv:local-character-function-on-object}
For a local character function $\Upsilon$ and an object $A$ we will write $\Upsilon(A)$ for $\Upsilon(\lambda)$ where $\lambda$ is the least cardinal such that $A$ is $\lambda$-presentable.
\end{convention}
\begin{definition}
\thlabel{def:3-amalgamation}
An independence relation has \textsc{3-amalgamation} if the following holds. Suppose that we have
\[
A \indep[M]{N_1} B, \quad
B \indep[M]{N_2} C, \quad
C \indep[M]{N_3} A,
\]
so $A$ is the domain of a subobject of $N_1$ and $N_3$, and similar for $B$ and $C$, while $M$ is considered a subobject of all three. Suppose furthermore that $M$ is a model and that
\begin{align*}
\gtp(a, m; N_1) &= \gtp(a, m; N_3), \\
\gtp(b, m; N_1) &= \gtp(b, m; N_2), \\
\gtp(c, m; N_2) &= \gtp(c, m; N_3),
\end{align*}
where $a$, $b$, $c$ and $m$ are representatives for the subobjects $A$, $B$, $C$ and $M$ respectively (overloading notation for subobjects of different models). Then we can find extensions from $N_1$, $N_2$ and $N_3$ to some $N$ such that the diagram we obtain in that way commutes:
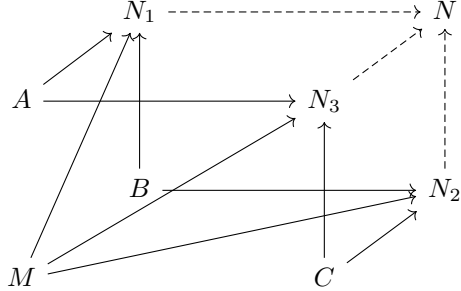
\[
\begin{tikzcd}
 & N_1 \arrow[rrr, dashed] &  &  & N \\
A \arrow[rrr] \arrow[ru] &  &  & N_3 \arrow[ru, dashed] &  \\
 & B \arrow[uu] \arrow[rrr] &  &  & N_2 \arrow[uu, dashed] \\
M \arrow[rrruu] \arrow[rrrru] \arrow[ruuu] &  &  & C \arrow[ru] \arrow[uu] & 
\end{tikzcd}
\]
Furthermore, these extensions are such that $A \indep[M]{N} N_2$.
\end{definition}
\begin{definition}
\thlabel{def:simple-and-stable-independence-relation}
Let $\forkindep$ be an independence relation that satisfies
\textsc{Invariance},
\textsc{Monotonicity},
\textsc{Base-Monotonicity},
\textsc{Transitivity},
\textsc{Symmetry},
\textsc{Existence},
\textsc{Union} and
\textsc{Local Character}.
If $\forkindep$ also satisfies \textsc{Stationarity}, then we call $\forkindep$ a \emph{stable independence relation}. If instead $\forkindep$ also satisfies \textsc{3-amalgamation}, then we call $\forkindep$ a \emph{simple independence relation}.
\end{definition}
\begin{remark}
\thlabel{rem:lrv-independence}
As opposed to \cite{lieberman_forking_2019} we have defined an independence relation here on triples of subobjects, while they define it as a relation on commuting squares. Their notion has the advantage of the independent squares forming an accessible category, and allowing for a more category-theoretic study of the independence relation itself (see also \cite{lieberman_cellular_2020}). Our approach has the benefit that the calculus we get from it is more intuitive and easier to work with.

In an AECat of the form $(\C, \C)$, these two notions are essentially the same. That is, assuming basic properties on the relevant independence relations, one can be recovered from the other and vice versa.
\end{remark}
\begin{example}
\thlabel{ex:non-finitely-short-union}
A natural question would be to ask whether there are examples that are not finitely short, but where there still is an independence relation satisfying \textsc{Union}. Quasiminimal excellent classes as discussed in \thref{ex:quasiminimal-excellent-classes} are such an example: the pregeometry there yields a stable independence relation.

Other examples can be found in AECs with intersections. For example \cite[Appendix C]{vasey_shelahs_2017} discusses how to find a stable independence relation in such AECs. Also \cite[Section 8.2]{grossberg_simple-like_2020} comes close to giving a simple example, rather than stable. Although no explicit examples are given and they do not get the full $(<\aleph_0)$-witness property.
\end{example}
\begin{proposition}
\thlabel{prop:basic-properties-independence-relation}
Let $\forkindep$ be an independence relation satisfying \textsc{Invariance}, \textsc{Monotonicity}, \textsc{Existence} and \textsc{Transitivity}, then the following hold:
\begin{enumerate}[label=(\roman*)]
\item for any $A, B \leq M$, we have $A \indep[B]{M} B$;
\item if we have $A \indep[C]{M} B$, then we can find an extension $M \to N$ and $D \leq N$ such that $B, C \leq D$ and $A \indep[C]{N} D$;
\item for $(a, b, c; M)$ such that $A \indep[C]{M} B$ and any $B' \leq M$ there is an extension $M \to N$ with some $a': A \to N$ such that $A' \indep[C]{N} B'$ and $\gtp(a', b, c; N) = \gtp(a, b, c; M)$.
\end{enumerate}
\end{proposition}
\begin{proof}
For (i), we use \textsc{Existence} to get an extension $M \to N$ and $A' \leq N$ such that $A' \indep[B]{N} B$ and $\gtp(A, B; M) = \gtp(A', B; N)$. Then \textsc{Invariance} yields the desired result.

For (ii), we use (i) to get $A \indep[C]{M} C$. Since we have by assumption that $A \indep[C]{M} B$, the result now directly follows from applying \textsc{Transitivity}.

Finally, for (iii) we use (ii) to get $M \to N'$ and $B,C \leq D \leq N'$ and $A \indep[C]{N'} D$. Then we use \textsc{Existence} to find $N' \to N$ and $a': A \to N$ such that $A' \indep[D]{N} B'$ and $\gtp(a', d; N) = \gtp(a, d; N')$. By \textsc{Invariance} then $A' \indep[C]{N'} D$ and thus $A' \indep[C]{N} B'$ by \textsc{Transitivity}. Also $\gtp(a', b, c; N) = \gtp(a, b, c; M)$, because $B, C \leq D$.
\end{proof}
\begin{proposition}
\thlabel{prop:local-character-sequence}
Let $\forkindep$ be an independence relation satisfying \textsc{Local Character} and \textsc{Base-Monotonicity}, with local character function $\Upsilon$. Let $(M_i)_{i < \kappa}$ be a chain of models, with chain bound $N$, and write $M_\kappa = \colim_{i < \kappa} M_i$. If $A$ is a subobject of $N$ such that $\kappa \geq \Upsilon(A)$, then there is some $i_0 < \kappa$ such that $A \indep[M_{i_0}]{N} M_\kappa$.
\end{proposition}
\begin{proof}
There is $\kappa$-presentable $M' \leq M_\kappa$ such that $A \indep[M']{N} M_\kappa$, by \textsc{Local Character}. Since $M_\kappa$ is the colimit of a $\kappa$-directed system, $M' \leq M_\kappa$ must factor as $M' \leq M_{i_0} \leq M_\kappa$ for some $i_0 < \kappa$. By \textsc{Base-Monotonicity} then $A \indep[M_{i_0}]{N} M_\kappa$.
\end{proof}
\begin{definition}
\thlabel{def:independent-sequence}
Suppose we have an independence relation $\forkindep$. Let $(a_i)_{i < \kappa}$ be a sequence in $M$ and let $c: C \to M$ be an arrow. Let $(M_i)_{i < \kappa}$ be a chain of initial segments for $(a_i)_{i < \kappa}$. Then we say that the $(M_i)_{i < \kappa}$ are \emph{witnesses of independence} for $(a_i)_{i < \kappa}$, if for all $i \in I$ we have
\[
A_i \indep[C]{M} M_i.
\]
Here $A_i$ is the subobject represented by $a_i: A_i \to M$, and likewise for $M_i \to M$.

We say that the sequence $(a_i)_{i < \kappa}$  is \emph{$\forkindep[C]$-independent} if it admits a chain of witnesses of independence.
\end{definition}
\begin{lemma}
\thlabel{lem:build-independent-isi-sequence}
Suppose that $\forkindep$ is an independence relation satisfying \textsc{Invariance}, \textsc{Monotonicity}, \textsc{Transitivity}, \textsc{Existence} and \textsc{Union}. Then given $(a, c; M)$ and any $\kappa$, there is a $\forkindep[C]$-independent isi-sequence $(a_i)_{i < \kappa}$ over $c$ in some extension $M \to N$ such that $\gtp(a_i, c; N) = \gtp(a, c; M)$ for all $i < \kappa$.
\end{lemma}
\begin{proof}
We inductively build a chains of models $(M_i)_{i < \kappa}$ and $(N_i)_{i < \kappa}$ together with arrows $a_i: A \to M_{i+1}$, where $N_0$ is an extension of $M$ and $C$ embeds in $(M_i)_{i < \kappa}$, such that:
\begin{enumerate}[label=(\roman*)]
\item there is an extension $M_i \to N_i$, and this is natural in the sense that
\[
\begin{tikzcd}
N_j \arrow[r] & N_i \\
M_j \arrow[r] \arrow[u] & M_i \arrow[u]
\end{tikzcd}
\]
commutes for all $j < i$;
\item $A \indep[C]{N_i} M_i$;
\item for successor $i = j+1$, we have $\gtp(a_j, m_j; N_i) = \gtp(a, m_j; N_i)$.
\end{enumerate}
\underline{\emph{Base case.}} Apply \textsc{Existence} to find $M \to N_0$ and $M_0 \leq N_0$ with $\gtp(m_0, c; N_0) = \gtp(m, c; M)$ and $A \indep[C]{N_0} M_0$.

\vspace{\baselineskip}\noindent
\underline{\emph{Successor step.}} We use the induction hypothesis to apply \thref{prop:basic-properties-independence-relation}(iii) to find and extension $N_i \to N_{i+1}$ and $M_{i+1} \leq N_{i+1}$ such that $A \indep[C]{N_{i+1}} M_{i+1}$ and $\gtp(m_{i+1}, m_i; N_{i+1}) = \gtp(Id_{N_i}, m_i; N_i)$. Properties (i) and (ii) follow directly. We had an arrow $a: A \to N_i$, and $M_{i+1}$ is the same object as $N_i$. So we have an arrow $a_{i+1}: A \to M_{i+1}$. Restricting the equality $\gtp(m_{i+1}, m_i; N_{i+1}) = \gtp(n_i, m_i; N_{i+1})$ then shows that property (iii) holds.

\vspace{\baselineskip}\noindent
\underline{\emph{Limit step.}} We let $M_i = \colim_{j < i} M_i$ and $N_i = \colim_{j < i} N_i$. For every $j < i$ we can compose $M_j \to N_j$ with the coprojection $N_j \to N_i$. By property (i) this makes $N_i$ into a cocone for $(M_j)_{j < i}$. So the universal property gives us an arrow $M_i \to N_i$, clearly satisfying property (i). Property (ii) follows from \textsc{Union}.

\vspace{\baselineskip}\noindent We set $N = \colim_{i < \kappa} N_i$. Then property (iii) ensures that $(a_i)_{i < \kappa}$ with chain of initial segments $(M_i)_{i < \kappa}$ is an isi-sequence. Since $C$ embeds in $(M_i)_{i < \kappa}$, we also see that $\gtp(a_i, c; N) = \gtp(a, c; M)$ for all $i < \kappa$. Finally, the $(M_i)_{i < \kappa}$ are witnesses of independence, which follows from combining properties (ii) and (iii).
\end{proof}
We close out this section with the following corollary, which immediately follows from \thref{prop:uniqueness-implies-3-amalgamation}.
\begin{corollary}
\thlabel{cor:stable-independence-relation-is-simple}
Every stable independence relation is also simple.
\end{corollary}
\begin{proposition}
\thlabel{prop:uniqueness-implies-3-amalgamation}
If an independence relation $\forkindep$ satisfies \textsc{Invariance}, \textsc{Transitivity}, \textsc{Existence}, \textsc{Symmetry} and \textsc{Stationarity}, then it also satisfies \textsc{3-amalgamation}.
\end{proposition}
\begin{proof}
Using the equalities $\gtp(b, m; N_1) = \gtp(b, m; N_2)$ and $\gtp(c, m; N_2) = \gtp(c, m; N_3)$ we find $U_1$ and $U_2$ as in the diagram below, which then commutes.
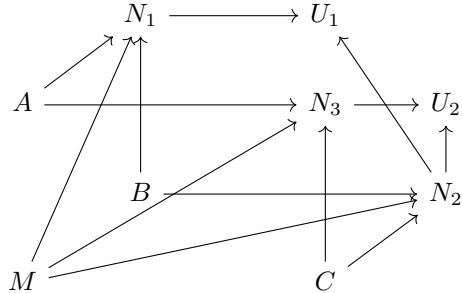
\[
\begin{tikzcd}
                                           & N_1 \arrow[rr]           &  & U_1                     &                           \\
A \arrow[rrr] \arrow[ru]                   &                          &  & N_3 \arrow[r]           & U_2                       \\
                                           & B \arrow[uu] \arrow[rrr] &  &                         & N_2 \arrow[u] \arrow[luu] \\
M \arrow[rrruu] \arrow[rrrru] \arrow[ruuu] &                          &  & C \arrow[ru] \arrow[uu] &                          
\end{tikzcd}
\]
We label the arrows with domain $A$ as $a_1: A \to U_1$ and $a_2: A \to U_2$.

By \thref{prop:basic-properties-independence-relation}(iii) we find an extension $U_1 \to U_1'$ and an arrow $n_2': N_2 \to U_1'$ such that $\gtp(n_2', b, m; U_1') = \gtp(n_2, b, m; U_1)$ and $A_1 \indep[M]{U_1'} N_2'$. This means that $B \to N_1 \to U_1'$ is the same arrow as $B \to N_2 \xrightarrow{n_2'} U_1'$. Similarly we find $n_2': N_2 \to U_2'$.

Since $\gtp(n_2', m; U_1') = \gtp(n_2', m; U_2')$ we can apply \textsc{Stationarity} to obtain $\gtp(a_1, n_2', m; U_1') = \gtp(a_2, n_2', m; U_2')$. This yields the dashed arrows in the diagram below. We left out the original arrows $N_2 \to U_1$ and $N_2 \to U_2$ because they are no longer relevant.
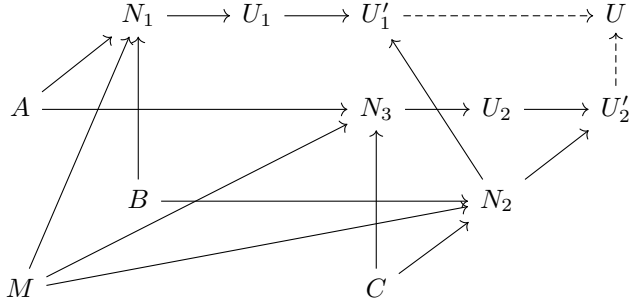
\[
\begin{tikzcd}
                                           & N_1 \arrow[r]            & U_1 \arrow[r] & U_1' \arrow[rr, dashed] &                            & U                      \\
A \arrow[rrr] \arrow[ru]                   &                          &               & N_3 \arrow[r]           & U_2 \arrow[r]              & U_2' \arrow[u, dashed] \\
                                           & B \arrow[uu] \arrow[rrr] &               &                         & N_2 \arrow[luu] \arrow[ru] &                        \\
M \arrow[rrruu] \arrow[rrrru] \arrow[ruuu] &                          &               & C \arrow[ru] \arrow[uu] &                            &                       
\end{tikzcd}
\]
This diagram commutes. So there is only one arrow $A \to U$, and by construction we have $A \indep[M]{U} N_2$. So $U$ is the required common extension of $N_1$, $N_2$ and $N_3$.
\end{proof}
\section{The Kim-Pillay theorem for AECats}
\label{sec:kim-pillay}
This section consists of the proof of our main theorem, a version of the Kim-Pillay theorem for AECats. The original first-order version of the theorem can be found as \cite[Theorem 4.2]{kim_simple_1997}. A more modern version appears as \cite[Theorem 7.3.13]{tent_course_2012}.

\thref{cor:unique-independence-relation} follows immediately from the main theorem. \thref{cor:stable-independence-relation-unique} follows directly from the main theorem and \thref{cor:stable-independence-relation-is-simple}. 

\begin{repeated-theorem}[\thref{thm:kim-pillay-category-theoretic}]
Let $(\C, \M)$ be an AECat with the amalgamation property, and suppose that $\forkindep$ is a simple independence relation. Let $A, B, C$ be subobjects of a model $M$. Then $A \indep[C]{M} B$ if and only if $\gtp(A, B, C; M)$ does not isi-divide over $C$.
\end{repeated-theorem}
\begin{proof}
Let $\forkindep$ and be a simple independence relation, and let $\Upsilon$ be a local character function. We will show that $A \indep[C]{D} B$ if and only if $\gtp(A, B, C; D)$ does not isi-divide over $C$. Here $D$ is some model.

\textbf{Isi-nondividing implies independence.} Let $a$, $b$ and $c$ be representatives of $A$, $B$ and $C$ respectively, such that $\gtp(a, b, c; D)$ does not isi-divide over $c$. By \thref{lem:build-independent-isi-sequence} we can construct a $\forkindep[C]$-independent isi-sequence $(b_i)_{i < \lambda}$ over $c$ in some $M$, for arbitrarily large $\lambda > \Upsilon(A)$, with witnesses of independence $(M_i)_{i < \lambda}$, such that
\[
\gtp(b_i, c; M) = \gtp(b, c; D)
\]
for all $i < \lambda$. If we pick the right $\lambda$ then by definition of isi-dividing there is a set $I \subseteq \lambda$ such that $|I| \geq \Upsilon(A)$ and $\gtp(a, b, c; D)$ is consistent for $(b_i)_{i \in I}$. So we may assume there is $a': A \to M$ such that for all $i \in I$:
\[
\gtp(a', b_i, c; M) = \gtp(a, b, c; D)
\]
By possibly deleting a tail segment from $I$ we may assume that $I$ has the order type of a cardinal $\geq \Upsilon(A)$. Let $M_I = \colim_{i \in I} M_i$ and consider $M_I$ as a subobject of $M$. Denote by $A'$ the subobject represented by $a'$, and use \thref{prop:local-character-sequence} to find $i_0 \in I$ such that $A' \indep[M_{i_0}]{M} M_I$.

We denote by $B_{i_0}$ the subobject of $M$ represented by $b_{i_0}$. By \textsc{Monotonicity} $A' \indep[M_{i_0}]{M} B_{i_0}$, because $B_{i_0} \leq M_I$, and hence
\[
B_{i_0} \indep[M_{i_0}]{M} A'
\]
by \textsc{Symmetry}. Furthermore, the fact that the $(M_i)_{i < \lambda}$ are witnesses of independence gives us
\[
B_{i_0} \indep[C]{M} M_{i_0}.
\]
Then by \textsc{Transitivity}, we find $B_{i_0} \indep[C]{M} A'$ and thus $A' \indep[C]{M} B_{i_0}$ by \textsc{Symmetry}. Since $a'$ is a realisation of $\gtp(a, b, c; D)$ for $(b_i)_{i \in I}$, we have $\gtp(A', B_{i_0}, C; M) = \gtp(A, B, C; D)$. So by \textsc{Invariance} we find $A \indep[C]{D} B$, as required.

\textbf{Independence implies isi-nondividing.} We now suppose that $A \indep[C]{D} B$. Fix representatives $a$, $b$ and $c$ of $A$, $B$ and $C$ respectively. Let $\lambda > \Upsilon(B)$ and $(b_i)_{i < \lambda}$ be an isi-sequence over $c$ in some $M$ with chain of initial segments $(M_i)_{i < \lambda}$ such that $\gtp(b_i, c; M) = \gtp(b, c; D)$ for all $i < \lambda$.

Let $\Upsilon(B) \leq \kappa < \lambda$. We have $M_\kappa = \colim_{i < \kappa} M_i$, so we can apply \thref{prop:local-character-sequence} to $M_\kappa$ and $B_\kappa$ considered as subobjects of $M$ to find $i_0 < \kappa$ such that $B_\kappa \indep[M_{i_0}]{M} M_\kappa$. We will show that $\gtp(a, b, c; D)$ is consistent for $(b_i)_{i_0 \leq i < \kappa}$.

\vspace{\baselineskip}\noindent
\emph{Claim:} for all $i_0 \leq i < \kappa$ we have $B_i \indep[M_{i_0}]{M} M_i$, and thus $B_i \indep[M_{i_0}]{M_{i+1}} M_i$.

\vspace{\baselineskip}\noindent
\emph{Proof of claim:} Since we have $B_\kappa \indep[M_{i_0}]{M} M_\kappa$, it follows from \textsc{Monotonicity} that $B_\kappa \indep[M_{i_0}]{M} M_i$. Then because $(b_i)_{i < \lambda}$ is an isi-sequence we have $\gtp(b_\kappa, m_i, c; M) = \gtp(b_i, m_i, c; M)$, so in particular we have $\gtp(B_\kappa, M_i; M) = \gtp(B_i, M_i; M)$. As $M_{i_0} \leq M_i$ we actually have that $\gtp(B_\kappa, M_i, M_{i_0}; M) = \gtp(B_i, M_i, M_{i_0}; M)$. The claim then follows from \textsc{Invariance}.

\vspace{\baselineskip}\noindent
We will now use the sequence $(b_i)_{i_0 \leq i < \kappa}$ to build a chain of models $(N_i)_{i_0 \leq i < \kappa}$, with monomorphisms $a': A \to N_{i_0}$, $b: B \to N_{i_0}$ and $c: C \to N_{i_0}$ such that $\gtp(a', b, c; N_{i_0}) = \gtp(a, b, c; D)$. So this is really saying that $A$, $B$ and $C$ are embedded in the chain $(N_i)_{i_0 \leq i < \kappa}$. The reason that we use the same notation for $b$ and $c$ as monomorphisms into $N_{i_0}$ and $D$, while we make a distinction between $a$ and $a'$, is because $N_{i_0}$ will be an extension of $D$ and $b$ and $c$ will just be the composition with this extension. On the other hand, $a'$ will not be the composition of $a: A \to D$ and the extension $D \to N_{i_0}$.

We construct this chain by transfinite induction, and such that at stage $i$:
\begin{enumerate}[label=(\roman*)]
\item there is an extension $m_i': M_i \to N_i$, and this is natural in the sense that
\[
\begin{tikzcd}
N_j \arrow[r] & N_i \\
M_j \arrow[r] \arrow[u, "m_j'"] & M_i \arrow[u, "m_i'"']
\end{tikzcd}
\]
commutes for all $i_0 \leq j < i$;
\item if $i$ is a successor (and is not $i_0$), say $i = j+1$, then $\gtp(a', b_j', c; N_i) = \gtp(a, b, c; D)$, where $b_j'$ is the composition $B \xrightarrow{b_j} M_i \xrightarrow{m_i'} N_i$;
\item $A' \indep[M_{i_0}]{N_i} M_i'$.
\end{enumerate}

\vspace{\baselineskip}\noindent
\underline{\emph{Base case, $i = i_0$.}} To build $N_{i_0}$ we first use $\gtp(b_{i_0}, c; M) = \gtp(b, c; D)$ to find extensions $M \to  N \leftarrow D$ witnessing this. Then we apply \thref{prop:basic-properties-independence-relation}(iii) to find an extension $N \to N_{i_0}$ and an arrow $a': A \to N_{i_0}$ such that $A' \indep[C]{N_{i_0}} N$ and $\gtp(a', b, c; N_{i_0}) = \gtp(a, b, c; N)$. We take $m_{i_0}'$ to be the arrow $M_{i_0} \to M \to N_{i_0}$, so $m_{i_0}' = m_{i_0}$. The properties in the induction hypothesis are trivial.

There are two more important properties of $N_{i_0}$ that will be used in the successor step.  The first one is that $A' \indep[M_{i_0}]{N_{i_0}} B$, by \textsc{Base-Monotonicity} and \textsc{Monotonicity}. The second one is:
\begin{align}
\label{eq:gtp-b-mi0}
\gtp(b, m_{i_0}; N_{i_0}) = \gtp(b_{i_0}, m_{i_0}; M) = \gtp(b_i, m_{i_0}; M) = \gtp(b_i, m_{i_0}; M_{i+1}).
\end{align}
The first equality follows because by construction $b_{i_0}: B \to M$ composed with the extension $M \to N_{i_0}$ and $b: B \to D$ composed with the extension $D \to N_{i_0}$ are the same arrow. The second equality follows because $(b_i)_{i < \lambda}$ is an isi-sequence. The last equality follows because the $(M_i)_{i < \lambda}$ form a chain of initial segments.

\vspace{\baselineskip}\noindent
\underline{\emph{Successor step.}} Suppose we have constructed $N_i$. By the claim earlier we have $B_i \indep[M_{i_0}]{M_{i+1}} M_i$, by construction we have $A' \indep[M_{i_0}]{N_{i_0}} B$ and finally we have $A' \indep[M_{i_0}]{N_i} M_i'$ from the induction hypothesis. We wish to apply \textsc{3-amalgamation} to this. For that we need to check that the following Galois types are equal:
\begin{itemize}
\item $\gtp(a', m_{i_0}; N_{i_0}) = \gtp(a', m_{i_0}; N_i)$, this holds because $N_i$ is just an extension of $N_{i_0}$;
\item $\gtp(b_i, m_{i_0}, M_{i+1}) = \gtp(b, m_{i_0}; N_{i_0})$, this is just the equality in (\ref{eq:gtp-b-mi0});
\item $\gtp(m_i, m_{i_0}; M_{i+1}) = \gtp(m_i', m_{i_0}; N_i)$, follows from the fact that $M_i$ is a model, so $\gtp(m_i; M_{i+1}) = \gtp(m_i; M_i) = \gtp(m_i'; N_i)$, and the fact that $m_{i_0}$ factors through $m_i$ and $m_i'$. Here property (i) of the induction hypothesis is important to guarantee that $M_{i_0} \to N_{i_0} \to N_i$ is really the same arrow as $M_{i_0} \to M_i \to N_i$.
\end{itemize}
So we can indeed apply \textsc{3-amalgamation} to find extensions from $M_{i+1}$, $N_{i_0}$ and $N_i$ to $N_{i+1}$:
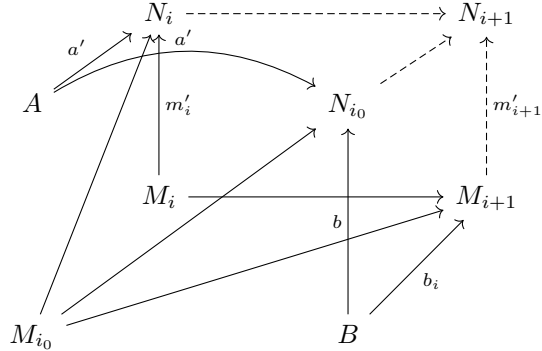
\[
\begin{tikzcd}
                                                    & N_i \arrow[rrr, dashed]             &  &                                        & N_{i+1}                                 \\
A \arrow[ru, "a'"] \arrow[rrr, "a'", bend left]     &                                     &  & N_{i_0} \arrow[ru, dashed]             &                                         \\
                                                    & M_i \arrow[uu, "m_i'"'] \arrow[rrr] &  &                                        & M_{i+1} \arrow[uu, "m_{i+1}'"', dashed] \\
                                                    &                                     &  &                                        &                                         \\
M_{i_0} \arrow[ruuuu] \arrow[rrruuu] \arrow[rrrruu] &                                     &  & B \arrow[ruu, "b_i"'] \arrow[uuu, "b"] &                                        
\end{tikzcd}
\]
We have to check the three properties of the induction hypothesis.
\begin{enumerate}[label=(\roman*)]
\item As a result of \textsc{3-amalgamation}, the square
\[
\begin{tikzcd}
N_i \arrow[r] & N_{i+1} \\
M_i \arrow[r] \arrow[u, "m_i'"] & M_{i+1} \arrow[u, "m_{i+1}'"']
\end{tikzcd}
\]
commutes. Because $M_j \to N_{i+1}$ will factor through $M_i$ for all $j < i$, and the induction hypothesis is satisfied for $i$, we see that in fact the naturality condition is satisfied for all $j < i+1$.
\item We have $\gtp(a', b_i', c; N_{i+1}) = \gtp(a', b, c; N_{i_0}) = \gtp(a, b, c; D)$.
\item By \textsc{3-amalgamation} we directly get $A' \indep[M_{i_0}]{N_{i+1}} M_{i+1}'$.
\end{enumerate}

\vspace{\baselineskip}\noindent
\underline{\emph{Limit step.}} We set $N_i = \colim_{i_0 \leq j < i} N_j$. By (i) from the induction hypothesis the arrows $m_j'$ composed with the coprojections $N_j \to N_i$ form a cocone on $(M_j)_{i_0 \leq j < i}$. By the universal property of the colimit $M_i = \colim_{i_0 \leq j < i} M_j$ we find the required extension $m_i': M_i \to N_i$. This shows that (i) is satisfied. Property (ii) is vacuous. And for (iii) we use the induction hypothesis to see that $A' \indep[M_{i_0}]{N_i} M_j'$ for all $i_0 \leq j < i$, and so we can apply \textsc{Union} to find $A' \indep[M_{i_0}]{N_i} M_i'$. This finishes the inductive construction of $(N_i)_{i_0 \leq i < \kappa}$.

\vspace{\baselineskip}\noindent
Now that we have constructed $(N_i)_{i_0 \leq i < \kappa}$ we can set $N_\kappa = \colim_{i_0 \leq i<\kappa} N_i$. Then for each $i_0 \leq i < \kappa$ we have by (ii) from the induction hypothesis that
\[
\gtp(a', b_i', c; N_\kappa) =
\gtp(a', b_i', c; N_{i+1}) =
\gtp(a, b, c; D).
\]
It follows from property (i) of the induction hypothesis that $N_\kappa$ is a chain bound for $(M_i)_{i_0 \leq i < \kappa}$. So by \thref{lem:chain-consistency} $\gtp(a, b, c; D)$ is consistent for $(b_i)_{i_0 \leq i < \kappa}$. As $(b_i)_{i_0 \leq i < \kappa}$ is a subsequence of size $\kappa$ of an arbitrarily long isi-sequence $(b_i)_{i < \lambda}$ over $c$, and $\kappa$ was arbitrarily large below $\lambda$, we conclude that $\gtp(a, b, c; D)$ does not isi-divide.
\end{proof}

\bibliographystyle{alpha}
\bibliography{bibfile}


\end{document}